\documentclass[12pt]{amsart}
\usepackage{amssymb,amsmath,amsthm}
\numberwithin{equation}{section}

\newtheorem{thm}{Theorem}[section]
\newtheorem{lemma}[thm]{Lemma}
\newtheorem{remark}[thm]{Remark}

\newtheorem{definition}[thm]{Definition}

\allowdisplaybreaks

\begin{document}

\title[Extreme Eigenvalues of Quaternion Sample Covariance Matrix] {Extreme Eigenvalues of Large Dimensional Quaternion Sample Covariance Matrix }

\author{HUIQIN LI,\ \ ZHIDONG BAI}
\thanks{ H. Q. Li was partially supported by a grant CNSF 11301063; Z. D. Bai was partially supported by CNSF  11171057, the Fundamental Research Funds for the Central Universities, PCSIRT, and the NUS Grant R-155-000-141-112.
}

\address{KLASMOE and School of Mathematics \& Statistics, Northeast Normal University, Changchun, P.R.C., 130024.}
\email{lihq118@nenu.edu.cn}
\address{KLASMOE and School of Mathematics \& Statistics, Northeast Normal University, Changchun, P.R.C., 130024.}
\email{baizd@nenu.edu.cn}


\maketitle

\begin{abstract}
In this paper, we shall investigate the almost sure limits of the largest and smallest eigenvalues of a quaternion sample covariance matrix. Suppose that $\mathbf X_n$ is a $p\times n$ matrix whose elements are independent quaternion variables with mean zero, variance 1 and uniformly bounded fourth moments. Denote $\mathbf S_n=\frac{1}{n}\mathbf X_n\mathbf X_n^*$. In this paper, we shall show that $s_{\max}\left(\mathbf S_n\right)=s_{p}\left(\mathbf S_n\right)\to\left(1+\sqrt y\right)^2, a.s.$ and $s_{\min}\left(\mathbf S_n\right)\to\left(1-\sqrt y\right)^2,a.s.$ as $n\to\infty$, where $y=\lim p/n$,
 $s_1\left(\mathbf S_n\right)\le\cdots\le s_{p}\left(\mathbf S_n\right)$ are the eigenvalues of $\mathbf{S}_n$,   $s_{\min}\left(\mathbf S_n\right)=s_{p-n+1}\left(\mathbf S_n\right)$ when $p>n$ and $s_{\min}\left(\mathbf S_n\right)=s_1\left(\mathbf S_n\right)$ when $p\le n$. We also prove that the set of  conditions  are necessary for $s_{\max}\left(\mathbf S_n\right)\to\left(1+\sqrt y\right)^2, a.s.$ when the entries of $\mathbf {X}_n$ are i. i. d.


\end{abstract}

\section{Introduction. }
 Let $\mathbf A$ be a $p \times p$ Hermitian matrix with eigenvalues ${s_j\left(\mathbf A\right)}, j = 1,2, \cdots, p$ arranged ascendingly, i.e.,  $s_1\left(\mathbf A\right)\le\cdots \le s_p\left(\mathbf A\right)$. Then the empirical spectral distribution (ESD) of the matrix $\mathbf A$ is defined by
$${F^{\mathbf A}}\left(x\right) =\frac{1}{p}\max\left\{j:{s _j\left(\mathbf A\right)\le x}\right\}.$$
If there is a sequence of random matrices whose ESD weakly converges to a limit, then the limit is said to be the LSD (Limiting Spectral Distribution) of the sequence of random matrices.

Eigenvalues of random matrix are often used in multivariate statistical analysis, such as the principal component analysis, multiple discriminant analysis, and canonical correlation analysis, etc. For example, many important statistics in multivariate statistical analysis are constructed by the eigenvalues of sample covariance matrices or those of multivariate $F$ matrices. Moreover, they can be written as functions of integrals with respect to the ESD of sample covariance matrices or multivariate $F$ matrices. When LSD is known, the corresponding functionals with respect to the LSD can be viewed as the population parameters and those respect to the ESD can be considered as the parameter estimators. Therefore, one may want to apply the Helly-Bray theorem to find the approximation of the statistics to their estimand. Unfortunately, the integrands are usually unbounded which leads to the failure of the application of the Helly-Bray theorem. Thus the limiting behavior of the extreme eigenvalues of sample covariance matrices or multivariate $F$ matrices is of special interest.

When the underlying random variables are real and/or complex, intensive work has been done in the literature (see \cite{geman1980limit,Yin1988,Bai1988166,silverstein1985smallest,BaiYin1993,bai1999methodologies,bai1998no,bai1987limiting}, among others). It is well known that the ESD of a sample covariance matrix $\mathbf W_n=\frac{1}{n}{\mathbf Y_n}{\mathbf Y_n^*}$ (the entries of $\mathbf Y_n =\left(y_{jl}\right)_{p\times n}$ are i.i.d. real random variables with mean zero and variance $\sigma^2$) converges to the M-P (Mar\v{c}enko-Pastur) law $F_y\left(x\right)$ with density
\begin{align*}
f_y\left(x\right)={\frac{1}{{2\pi xy{\sigma ^2}}}\sqrt {\left(b - x\right)\left(x -a\right)}}I_{[a,b]}\left(x\right)+I_{(1,\infty)}\left(y\right)\left(1-y^{-1}\right)\delta\left(x\right)
\end{align*}
where $a = {\sigma ^2}{(1 - \sqrt y )^2}$, $b = {\sigma ^2}{(1 + \sqrt y )^2}$ and $y =\lim p/n\in\left(0,\infty\right) $. Here $\delta\left(x\right)$ denotes the Dirac delta function and $I_{[a,b]}\left(x\right)$ denotes the indicator function of the interval $[a,b]$. Denote the eigenvalues of $\mathbf W_n$ by $s_1\left(\mathbf W_n\right),\cdots,s_{p}\left(\mathbf W_n\right)$, arranged in ascending  order. For the convergence of $s_{p}\left(\mathbf W_n\right)$,
Yin, Bai and Krishnaiah (1988) \cite{Yin1988} proved that $s_{p}\left(\mathbf W_n\right)\to\sigma^2\left(1+\sqrt y\right)^2, a.s.$ under the condition that
\begin{align*}
{\rm E}\left|y_{11}\right|^4<\infty.
\end{align*}  Moreover, Bai, Silverstein and Yin (1988) \cite{Bai1988166} showed that finite fourth moment is also necessary for the strong convergence of the largest eigenvalue. Therefore, we obtain the sufficient and necessary conditions of the strong convergence of the largest eigenvalue of $\mathbf W_n$. For the convergence of the smallest eigenvalue, we need to make the following declaration:
\begin{align*}
s_{\min}\left(\mathbf W_n\right)=
\begin{cases}
s_{\min}\left(\mathbf W_n\right)=s_1\left(\mathbf W_n\right)&p\le n,\cr
s_{\min}\left(\mathbf W_n\right)=s_{p-n+1}\left(\mathbf W_n\right)&p>n.\cr\end{cases}
\end{align*}
Bai and Yin (1993) \cite{BaiYin1993} proved that
\begin{align*}
s_{\min}\left(\mathbf W_n\right)\to\sigma^2\left(1-\sqrt y\right)^2,   a.s.
\end{align*}
where the underlying distribution has a zero mean and finite fourth moment. The results above were extended to the complex case in \cite{bai1999methodologies}. In this paper, we shall show that the conclusions are still true for the quaternion sample covariance matrix.

Next we introduce some notations and some basic properties about quaternions. The quaternion base can be represented by four $2\times 2$ matrices as
\begin{align*}
\mathbf {e} = \left( \begin{array}{cc}
1&0\\
0&1\\
\end{array} \right),
\mathbf i = \left( \begin{array}{cc}
i&0\\
0&- i\\
\end{array} \right),
\mathbf j = \left( \begin{array}{cc}
0&1\\
-1&0\\
\end{array}\right),
\mathbf k = \left( \begin{array}{cc}
0&i\\
i&0\\
\end{array}\right),\end{align*}
where $i=\sqrt{-1}$ denotes the imaginary unit. Thus, a quaternion can be represented by a $2\times 2$ complex matrix as
\begin{align*}
x = a \cdot \mathbf e + b \cdot \mathbf i + c \cdot \mathbf j + d \cdot \mathbf k =\left( {\begin{array}{*{20}{c}}
a+bi &c+di\\
{ - c+di }&{a-bi }
\end{array}} \right)
=\left( {\begin{array}{*{20}{c}}
\lambda &\omega\\
-\overline{\omega }&\overline{\lambda}
\end{array}} \right)
\end{align*}
where the coefficients  $a,b,c,d$ are real and $\lambda=a+bi,\omega=c+di$. The conjugate of $x$ is defined as
$$\bar x = a \cdot \mathbf e - b \cdot \mathbf i - c \cdot \mathbf j - d \cdot \mathbf k=\left( {\begin{array}{*{20}{c}}
a-bi &-c-di \\
{ c-di }&{a+bi }
\end{array}} \right)
=\left( {\begin{array}{*{20}{c}}
\overline{\lambda} &-\omega\\
\overline{\omega }&\lambda
\end{array}} \right) $$
and its norm as
 $$\left\| x \right\| = \sqrt {{a^2} + {b^2} + {c^2} + {d^2}}=\sqrt{\left|\lambda\right|^2+\left|\omega\right|^2}.$$ By the property of quaternions, one has
 \begin{align}\label{al:11}
 {\rm det}\left(x\right)=\left\|x\right\|^2.
  \end{align}
Furthermore, let $\mathbf I_p^Q$ denote $p\times p$ quaternion identity matrix, i. e.,
\begin{align*}
\mathbf I_p^Q={\rm diag}\left(\overbrace{\mathbf e,\cdots,\mathbf e}^p\right).
\end{align*}
 More details can be found in \cite{adler1995quaternionic,finkelstein1962foundations,zhang1995numerical,kuipers1999quaternions,mehta2004random,zhang1997quaternions,zhang1994numerical}. It is worth mentioning that any \ $n\times n$ quaternion matrix \ $\mathbf Y$ can be represented by a \ $2n\times 2n$ complex matrix \ $\psi(\mathbf Y)$. Consequently, we can deal with quaternion matrices as complex matrices for convenience. It is known (see \cite{zhang1997quaternions}) that the multiplicities of all the eigenvalues (obviously they are all real) of $\psi(\mathbf Y)$ are even. Taking one from each of the $n$ pairs of eigenvalues of $\psi(\mathbf Y)$, the $n$ values are defined as the eigenvalues of $\mathbf Y$.

This paper is organized as follows. The main theorems are stated in Section 2. In Section 3, we outline some knowledges of graph theory and introduce an operation called ``Diamond product" which will be used in Section 4. Section 4, Section 5, and Section 6 give the proofs of the main theorems, respectively. Some technical  lemmas are postponed to Section 7.

\section{ Main Theorem.}

In  this paper, we consider the strong limits of the largest and smallest eigenvalues of quaternion sample covariance matrices. Let $$ \mathbf S_n=\frac{1}{n}{\mathbf X_n}{\mathbf X_n^*}$$ where $\mathbf X_n$ is defined in Theorem \ref{th:1} and denote the eigenvalues of $\mathbf S_n$ by $s_1\left(\mathbf S_n\right),\cdots,s_{p}\left(\mathbf S_n\right)$, arranged in ascending order. Firstly, we give the upper and lower bounds of extreme eigenvalues in Theorem \ref{th:1} when $y=\lim p/n\in\left(0,1\right)$. Combining Theorem 1.1 $\left(F^{\mathbf S_n}\rightarrow F_y,{ a.s.}\right)$ in \cite{li2013convergence}, we can get Theorem \ref{th:2} about the limits of the largest and smallest eigenvalues while $y\in\left(0,1\right)$. Considering that ${\mathbf X_n}{\mathbf X_n^*}$ and $ {\mathbf X_n^*}{\mathbf X_n}$ have the same set of nonzero eigenvalues, Theorem \ref{th:2} is still true for $y\in \left(1,\infty\right)$. Finally, we present sufficient and necessary conditions for the existence of the strong limit of the largest eigenvalue of $\mathbf S_n$. These theorems can be stated as the following:
\begin{thm}\label{th:1}
Let $ \mathbf S_n=\frac{1}{n}{\mathbf X_n}{\mathbf X_n^*}$ where $\mathbf X_n = \left(x_{jl},j = 1, \cdots ,p, l= 1, \cdots ,n\right)$ and $x_{jl}$ are quaternion variables. Assume that the following conditions hold:
\begin{enumerate}
  \item  $x_{jl}$ are independent,
  \item  ${\rm E}x_{jl}=0$ and ${\rm Var}x_{jl}=\sigma^2$, for all $j,l$,
  \item  $\sup_{jl}{\rm E}\left\|x_{jl}\right\|^4\le M$, $M$ is a positive constant,
  \item  there exists a random variable $\xi$ with finite 4th moment and a constant $L$ such that for any $\delta>0$, \begin{align}\label{al:17}
      \frac{1}{np}\sum_{jl}{\rm P}\left(\left\|x_{jl}\right\|>\delta\sqrt n\right)<L{\rm P}\left(\left|\xi\right|>\delta\sqrt n\right).
      \end{align}
\end{enumerate}
Then we have
\begin{align*}
-2\sqrt y\sigma^2\le& \liminf s_1\left(\mathbf S_n-\sigma^2\left(1+y\right)\mathbf I_{p}^Q\right)\\
\le&\limsup s_{p}\left(\mathbf S_n-\sigma^2\left(1+y\right)\mathbf I_{p}^Q\right)\le2\sqrt y\sigma^2 \quad a.s.
\end{align*}
as $n\to\infty,p\to\infty,p/n\to y\in\left(0,1\right)$.
\end{thm}
From Theorem \ref{th:1}, one can easily get the following theorem:
\begin{thm}\label{th:2}
Under the conditions of Theorem \ref{th:1}, we have
\begin{align}
\lim s_{1}\left(\mathbf S_n\right)=\left(1-\sqrt y\right)^2\sigma^2\quad a.s.\label{al:1}\\
\lim s_{p}\left(\mathbf S_n\right)=\left(1+\sqrt y\right)^2\sigma^2\quad a.s.\label{al:2}
\end{align}
as $n\to\infty,p\to\infty,p/n\to y\in\left(0,1\right)$.
\end{thm}

\begin{remark}\label{re:2}
(\ref{al:1}) and (\ref{al:2}) are trivially true for $y=1$. 
 If $p>n$, one has that the $p-n$ smallest eigenvalues of $\mathbf S_n$ must be zero. Define
 \begin{align*}
 s_{\min}\left(\mathbf S_n\right)=
 \begin{cases}
 s_{p-n+1}\left(\mathbf S_n\right)&p>n,\cr
 s_1\left(\mathbf S_n\right)&p\le n.\cr\end{cases}
 \end{align*}
 We assert that Theorem \ref{th:2} is still true for  $y\in \left(1,\infty\right)$. In fact, when $y>1$, let $\breve{\mathbf S}_n=\frac{1}{p}\mathbf X_n^*\mathbf X_n$ and $\breve y=1/y\in\left(0,1\right)$. Applying Theorem \ref{th:2}, we have
\begin{align*}
\lim s_{\min}\left(\breve{\mathbf S}_n\right)=\left(1-\sqrt {\breve y}\right)^2\sigma^2=\frac{1}{ y}\left(1-\sqrt { y}\right)^2\sigma^2\quad a.s.\\
\lim s_{p}\left(\breve{\mathbf S}_n\right)=\left(1+\sqrt {\breve y}\right)^2\sigma^2=\frac{1}{ y}\left(1+\sqrt { y}\right)^2\sigma^2\quad a.s.
\end{align*}
which implies that
\begin{align*}
&\lim s_{\min}\left({\mathbf S_n}\right)=\frac{p}{n}\lim s_{\min}\left(\breve{\mathbf S}_n\right)=\left(1-\sqrt {y}\right)^2\sigma^2\quad a.s.\\
&\lim s_{p}\left({\mathbf S_n}\right)=\frac{p}{n}\lim s_{p}\left(\breve{\mathbf S}_n\right)=\left(1+\sqrt { y}\right)^2\sigma^2\quad a.s..
\end{align*}
\end{remark}

\begin{thm}\label{th:3}
Suppose that the entries of $\mathbf Z_n$ are i.i.d. quaternion random variables and the ratio of dimension to sample size $p/n\to y$, then the largest eigenvalue of $\mathbf \Lambda_n=\frac{1}{n}\mathbf{Z_nZ_n^*}$ tends to $\mu$ with probability 1 if and only if the following conditions are true:
\begin{align*}
({\rm i}) &\quad{\rm E} z_{11}=0;\\
({\rm ii})&\quad{\rm Var} z_{11}=\sigma^2;\\
({\rm iii})&\quad {\rm E}\left\| z_{11}\right\|^4<\infty;\\
({\rm iv})&\quad\mu=\left(1+\sqrt { y}\right)^2\sigma^2.
\end{align*}
\end{thm}
\section{Preliminaries.}

In this section, we will recall some basic knowledges of the graph theory (see Section 3.1.2 or Section 5.2 in \cite{bai2010spectral}) and introduce an operation of matrices.
\subsection{Some knowledges of Graph Theory.}\label{se:1}
Suppose that $i_1,\cdots,i_k$ are $k$ positive integers (not necessarily distinct) not
greater than $p$ and $j_1,\cdots,j_k$ are $k$ positive integers (not necessarily distinct)
not larger than $n$. For a sequence $\left(i_1,j_1,\cdots,i_k,j_k\right)$, draw two parallel lines,
referring to the $I$ line and the $J$ line. Plot $i_1,\cdots,i_k$ on the $I$ line and $j_1,\cdots,j_k$
on the $J$ line, and draw $k$ (down) edges from $i_u$ to $j_u$, $u = 1,\cdots, k$ and $k$ (up) edges from $j_u$ to $i_{u+1}, u = 1,\cdots, k$ (with the convention that $i_{k+1}=i_1$). The graph is denoted by $\mathbf{G\left(\mathcal{I}, \mathcal{J}\right)}$, where $\mathcal{I}= \left(i_1,\cdots, i_k\right)$ and $\mathcal{J}= \left(j_1,\cdots, j_k\right).$

Suppose the number of noncoincident $I$ -vertices is $r + 1$ and the number of noncoincident $J$ -vertices is $s$. A canonical graph can be defined as follows:
\begin{definition}\label{de:2}
A canonical $\Delta\left(k, r, s\right)$ can be directly defined in the following way:\\
1. Its vertex set $V = V_I +V_J$, where $V_I = \left\{1, \cdots, r +1\right\}$, called the $I$-vertices, and $V_J = \left\{1, \cdots, s\right\}$, called the $J$-vertices.\\
2. There are two functions, $f : \left\{1,\cdots, k\right\} \rightarrow \left\{1, \cdots, r+1\right\}$ and $g  : \left\{1, \cdots, k\right\}\rightarrow\left\{1, \cdots, s\right\}$, satisfying
\begin{align*}
&f\left(1\right) = 1 = g\left(1\right) = f\left(k+1\right),\\
&f\left(j\right)\le \max\left\{f\left(1\right), \cdots, f\left(j-1\right)\right\} + 1,\\
&g\left(j\right) \le \max\left\{g\left(1\right), \cdots, g\left(j-1\right)\right\} + 1.
\end{align*}
3. Its edge set $E = \left\{e_{1d}, e_{1u}, \cdots, e_{kd}, e_{ku}\right\}$, where $e_{1d}, \cdots, e_{kd}$ are called the
down edges and $e_{1u}, \cdots, e_{ku}$ are called the up edges.\\
4. $F \left(e_{jd}\right) = \left(f \left(j \right), g \left(j \right)\right)$ and $F \left(e_{ju} \right) = \left(g\left(j \right), f \left(j + 1\right)\right)$ for $j = 1, \cdots, k$.
\end{definition}
\begin{remark}
Two graphs are said to be isomorphic if one becomes the other by a suitable permutation on $\left(1, 2,\cdots, p\right)$ and a permutation on $\left(1, 2, \cdots, n\right)$. By Definition \ref{de:2}, we can easily obtain that there is only one canonical graph for each isomorphic class.
\end{remark}
\begin{remark}\label{re:1}
By Definition \ref{de:2}, the number of graphs in the isomorphic class associated with the canonical $\Delta\left(k, r, s\right)$-graph is
\begin{align*}
p\left(p-1\right)\cdots\left(p-r\right)n\left(n-1\right)\cdots\left(n-s+1\right)=p^{r+1}n^s\left[1+O\left(n^{-1}\right)\right].
\end{align*}
\end{remark}
If two edges have the same vertex sets, we say that the two edges coincide. We call that an edge $e_a$ is single up to $e_b,b\ge a$, when the edge $e_a$ does not coincide with any one among $e_1,\cdots ,e_b$ other than itself.
\begin{definition}
For a canonical graph, classify the edges into several types:\\
1. If $f\left(j+1\right)=\max\left\{f\left(1\right),\cdots,f\left(j\right)\right\}+1$, the edge $e_{ju}=\left(g\left(j\right),f\left(j+1\right)\right)$ is called an up innovation. And  if $g\left(j\right)=\max\left\{g\left(1\right),\cdots,g\left(j-1\right)\right\}+1$, the edge $e_{jd}=\left(f\left(j\right),g\left(j\right)\right)$ is called a down innovation. The two cases are both called a  ${\rm\mathcal{\mathbf T}_1}$ edge which leads to a new vertex.\\
2. An edge is called a ${\rm \mathbf{T}_3}$ edge if it coincides with an innovation that is single
until the ${\rm\mathbf{T}_3}$ edge appears. A ${\rm\mathbf{T}_3}$ edge $\left(g\left(j\right),f\left(j+1\right)\right)$ (or $\left(f\left(j\right),g\left(j\right)\right)$) is said to be irregular if there is only one innovation single up to $g\left(j\right)$ (or $f\left(j\right)$ ). All other ${\rm\mathbf{T}_3}$ edges are called regular ${\rm \mathbf{T}_3}$ edges.\\
3. All other edges are called ${\rm\mathbf{T}_4}$ edges.\\
4. The first appearance of a ${\rm \mathbf{T}_4}$ edge is called a ${\rm\mathbf{T}_2}$ edge. There are two cases: the first is the first appearance of a single noninnovation, and the second is the first appearance of an edge that coincides with a ${\rm\mathbf{T}_3}$ edge.
\end{definition}
A chain is a consecutive segment of $\mathbf{G\left(\mathcal{I}, \mathcal{J}\right)}$, i.e. $i_1j_1\cdots i_{\tau}$ ${\rm or} \ i_1j_1\cdots i_{\tau}j_{\tau}$.
\begin{lemma}\label{le:1}
Let $t$ denote the number of ${\rm\mathbf{T}_2}$ edges and $l$ denote the number of innovations in the chain $i_1j_1\cdots i_{\tau}$ $\left({\rm or} \ i_1j_1\cdots i_{\tau}j_{\tau}\right)$ that are single up to ${\tau}$ and have a vertex coincident with $f\left({\tau}\right)$ $\left({\rm or} \ g\left({\tau}\right)\right)$. Then $l\le t+1$.
\end{lemma}
\begin{lemma}\label{le:2}
The number of regular ${\rm\mathbf{T}_3}$ edges is not greater than twice the number of ${\rm \mathbf{T}_2}$ edges.
\end{lemma}
\subsection{{\bf Diamond product}}
\begin{definition}\label{de:3}
Let $\mathbf A$ and $\mathbf B$ be two $p\times n$ quaternion matrices. Then, $\mathbf A\star \mathbf B=\left(a_{jl}b_{jl}\right)$, called Hadamard product for quaternion matrices. 
\end{definition}
\begin{lemma}\label{le:9}
Let $\mathbf A$ and $\mathbf B$ be two $p\times n$ quaternion matrices. Then,
\begin{align*}
\left\|\mathbf A\star\mathbf B\right\|_2\le\left\|\mathbf A\right\|_2\left\|\mathbf B\right\|_2
\end{align*}
where $\left\|\cdot\right\|_2$ denotes the $2$-norm of a matrix, i.e. $\left\|\cdot\right\|_2$ is equal to the maximum singular value of this matrix.
\end{lemma}
\begin{proof}
Applying Definition \ref{de:3} and Lemma \ref{le:8}, we have
\begin{align*}
&\left\|\mathbf A\star\mathbf B\right\|_2\\
=&\sup_{\mathbf{e^*e=f^*f=1}}\left|\mathbf{e^*\left(\mathbf A\star\mathbf B\right)f}\right|\\
\le&\sup_{\mathbf{e^*e=f^*f=1}}\sum_{l=1}^p\sum_{k=1}^n\left|e_l^*\left(a_{lk}b_{lk}\right)f_k\right|\\
\le&\sup_{\mathbf{e^*e=1}}\left(\sum_{l=1}^p\sum_{k=1}^ne_l^*a_{lk}a_{lk}^*e_l\right)^{1/2}
\sup_{\mathbf{f^*f=1}}\left(\sum_{l=1}^p\sum_{k=1}^nf_k^*b_{lk}^{*}b_{lk}f_k\right)^{1/2}\\
\le&\sup_{\mathbf{e^*e=1}}\left(\max_l\sum_{k=1}^n\left\|a_{lk}\right\|^2\sum_{l=1}^pe_l^*e_l\right)^{1/2}
\sup_{\mathbf{f^*f=1}}\left(\max_k\sum_{l=1}^p\left\|b_{lk}\right\|^2\sum_{k=1}^nf_k^*f_k\right)^{1/2}\\
\le&\left\|\mathbf A\right\|_2\left\|\mathbf B\right\|_2
\end{align*}
where $e=\left(e_1',\cdots,e_p'\right)'$, $e_l$ is a vector of order $2$ and $f=\left(f_1',\cdots,f_n'\right)'$, $f_k$ is a vector of order $2$.
\end{proof}
\begin{definition}\label{de:1}
Let $\mathbf H_j = \left(h_{\alpha\beta}^{\left(j\right)}\right), j = 1,2,\cdots,k,$ be $k$ quaternion matrices with dimensions $n_j\times n_{j+1}$, respectively. Define the {\bf Diamond product} of the $k$ matrices by
\begin{align*}
\mathbf H_1\diamond\cdots\diamond\mathbf H_{k}=\left(\sum{h_{\alpha t_2}^{\left(1\right)}h_{t_2 t_3}^{\left(2\right)}\cdots h_{t_{k-1}t_{k}}^{\left(k-1\right)}h_{ t_{k}\beta}^{\left(k\right)}}\right)
\end{align*}
where the summation runs for $t_{j}=1,2,\cdots,n_j, \ j=2,\cdots,k$, subject to
restrictions $\alpha\neq t_3,t_2\neq t_4 ,t_3\neq t_5,\cdots,t_{k-2}\neq t_{k}$ and $t_{k-1}\neq \beta $.
\end{definition}

\begin{lemma}\label{le:5}
Let $\mathbf H_j = \left(h_{\alpha\beta}^{\left(j\right)}\right), j = 1,2,\cdots,k,$ be $k$ quaternion matrices with dimensions $n_j\times n_{j+1}$, respectively. Then, we have
\begin{align*}
\left\|\mathbf H_1\diamond\cdots\diamond\mathbf H_{k}\right\|_2\le3^{k-1}\left\|\mathbf H_1\right\|_2\cdots\left\|\mathbf H_{k}\right\|_2.
\end{align*}

\end{lemma}


\begin{proof}
We shall use induction to prove this lemma.
\begin{itemize}
\item{i)} When $k=1$, the conclusion is trivially true. When $k=2$, denote $\mathbf H_1\mathbf H_2=\left(\mathbf Q_{jl}\right)$ where $\mathbf Q_{jl}=\left(\begin{array}{cc} \lambda_{jl}&\omega_{jl}\\-\overline{\omega_{jl}}&\overline{\lambda_{jl}} \end{array}\right)$.
    It follows that
\begin{align*}
\left\|\mathbf H_1\diamond\mathbf H_2\right\|_2=&\left\|\mathbf H_1\mathbf H_2-{\rm diag}\left(\mathbf Q_{11},\cdots,\mathbf Q_{pp}\right)\right\|_2\\
\le&\left\|\mathbf H_1\mathbf H_2\right\|_2+\left\|{\rm diag}\left(\mathbf Q_{11},\cdots,\mathbf Q_{pp}\right)\right\|_2\\
\le&\left\|\mathbf H_1\mathbf H_2\right\|_2+\left\|{\rm diag}\left(\mathbf H_1\mathbf H_2\right)\right\|_2+\max_{j}\left|w_{jj}\right|\\
\le&3\left\|\mathbf H_1\right\|_2\left\|\mathbf H_2\right\|_2.
\end{align*}
\item{ii)} Let $k>1$. Note that
\begin{align*}
&\mathbf H_1\diamond\cdots\diamond\mathbf H_{k}\\
=&\mathbf H_1\left(\mathbf H_2\diamond\cdots\diamond\mathbf H_{k}\right)-{\rm diag}\left(\mathbf Q_{11},\cdots,\mathbf Q_{pp}\right)\left(\mathbf H_3\diamond\cdots\diamond\mathbf H_{k}\right)\\
&+\left(h_{jl}^{\left(1\right)}h_{lj}^{\left(2\right)}h_{jl}^{\left(3\right)}\right)\diamond\mathbf H_4\diamond\cdots\diamond\mathbf H_{k}.
\end{align*}
Here, the $\left(j,l\right)$ entry of the matrix $\left(h_{jl}^{\left(1\right)}h_{lj}^{\left(2\right)}h_{jl}^{\left(3\right)}\right)$ is zero if $l>n_2$ or $j>n_3$. Using Lemma \ref{le:9}, one has $$\left\|\left(h_{jl}^{\left(1\right)}h_{lj}^{\left(2\right)}h_{jl}^{\left(3\right)}\right)\right\|_2
\le\left\|\mathbf H_1\right\|_2\left\|\mathbf H_2'\right\|_2\left\|\mathbf H_3\right\|_2=\left\|\mathbf H_1\right\|_2\left\|\mathbf H_2\right\|_2\left\|\mathbf H_3\right\|_2.$$
\end{itemize}
By induction, we complete the proof of Lemma \ref{le:5}.
\end{proof}
\section{Proof of Theorem \ref{th:1}.}\label{se:3}

By Definition \ref{de:1}, we denote
\begin{align}\label{al:9}
\mathbf R_n\left(l\right)=n^{-l}\overbrace{\mathbf X_n\diamond\mathbf X_n^*\diamond\cdots\diamond\mathbf X_n}^{l \ \mathbf X_n}\diamond\mathbf X_n^*
\end{align}
which implies that $\mathbf R_n\left(1\right)=\mathbf S_n-n^{-1}{\rm diag}\left(\mathbf X_n\mathbf X_n^*\right)$. Thus, we shall complete the proof by the following two steps:
\begin{itemize}
\item a) Firstly, we derive the estimate of the norm of $\mathbf R_n\left(l\right)$. The aim of subsequent lemmas is to estimate of the norm of $\left(\mathbf R_n-y\sigma^2\mathbf I_{p}^Q\right)^k$  by using the estimate on $\mathbf R_n\left(l\right)$ (see Section 4.1);
 \item b) Applying these lemmas, we can easily get the bound of $\|\mathbf R_n-y\sigma^2\mathbf I_{p}^Q\|_2$. Together with $\left\|\mathbf S_n-\mathbf R_n-\sigma^2 \mathbf I_{p}^Q\right\|_2$, we obtain the bound of $\left\|\mathbf S_n-\sigma^2\left(1+y\right)\mathbf I_{p}^Q\right\|_2$ (see Section 4.2).
\end{itemize}

\subsection{Some lemmas}\label{se:2}
\begin{lemma}\label{le:3}
Under the conditions of Theorem \ref{th:1}, we have
\begin{align}\label{al:3}
\limsup_{n\to\infty}\left\|\mathbf R_n\left(l\right)\right\|_2\le\left(2l+1\right)\left(l+1\right)y^{\left(l-1\right)/2}\sigma^{2l}
\quad a.s..
\end{align}
\end{lemma}
\begin{proof}
By (\ref{al:9}),
\begin{align*}
\mathbf R_n\left(l\right)=n^{-l}\left(\sum x_{av_1}
\overline x_{u_1v_1}x_{u_1v_2}\overline x_{u_2v_2}\cdots x_{u_{l-1}v_l}\overline x_{bv_l}\right)
\end{align*}
where the summation $\sum$ runs over for $v_1,\cdots,v_l=1,\cdots,n$ and $u_1,\cdots,u_{l-1}=1,\cdots,p$ subject to the restriction $a\neq u_1,u_1\neq u_2,\cdots,u_{l-1}\neq b$ and $v_1\neq v_2,v_2\neq v_3,\cdots,v_{l-1}\neq v_l$.

Without loss of generality, we assume $\sigma=1$. At first, we will truncate and centralize the quaternion random variables without changing the bound of $\|\mathbf R_n\left(l\right)\|_2$.

Since ${\rm E}\left|\xi\right|^4<\infty$,  for any $\delta > 0$, we have
\begin{align*}
\sum_{k=1}^{\infty}\delta^{-4}2^{2k}{\rm P}\left(\left|\xi\right|>\delta2^{k/2}\right)<\infty.
\end{align*}
Then, we can select a slowly decreasing sequence of constants $\delta_{2^k}\to 0$, $2^{k/2}\delta_{2^k}\uparrow \infty$, and  such that
\begin{align}\label{al:10}
\sum_{k=1}^{\infty}\delta_{2^k}^{-4}2^{2k}{\rm P}\left(\left|\xi\right|>\delta_{2^k}2^{k/2}\right)<\infty.
\end{align}
Let $\delta_n=\delta_{2^k}$ for $2^k<n\le 2^{k+1}$ and let $\widehat x_{uv}=x_{uv}I\left(\left\|x_{uv}\right\|\le \delta_n\sqrt n\right)$ ($\delta_n\sqrt n\uparrow \infty$),  $\widehat{\mathbf X}_n=\left(\widehat x_{uv}\right)$, and
$$\widehat {\mathbf R}_n\left(l\right)=n^{-l}\overbrace{{\widehat{\mathbf X}_n\diamond\widehat{\mathbf X}_n^*\diamond\cdots\diamond\widehat{\mathbf X}_n}}^{l\ \widehat{\mathbf X}_n}\diamond\widehat{\mathbf X}_n^*.$$
Together with (\ref{al:17}) and (\ref{al:10}), one has
\begin{align*}
&{\rm P}\left(\widehat{\mathbf R}_n\neq\mathbf R_n,{\rm i.o.}\right)\\
=&\lim_{K\to\infty}{\rm P}\left(\bigcup_{k=K}^{\infty}\bigcup_{2^k< n\le2^{k+1}}\bigcup_{u\le p,v\le n}\left(\left\|x_{uv}\right\|>\delta_n \sqrt n\right)\right)\\
\le&\lim_{K\to\infty}\sum_{k=K}^{\infty}{\rm P}\left(\bigcup_{2^k< n\le2^{k+1}}\bigcup_{u\le (y+1)2^{k+1},v\le 2^{k+1}}\left(\left\|x_{uv}\right\|>\delta_{2^k} 2^{k/2}\right)\right)\\
=&\lim_{K\to\infty}\sum_{k=K}^{\infty}{\rm P}\left(\bigcup_{u\le (y+1)2^{k+1},v\le 2^{k+1}}\left(\left\|x_{uv}\right\|>\delta_{2^k} 2^{k/2}\right)\right)\\
\le&\lim_{K\to\infty}L(y+1)\sum_{k=K}^{\infty}2^{2k+2}{\rm P}\left(\left|\xi\right|>\delta_{2^k} 2^{k/2}\right)\to 0.
\end{align*}
Thus we only need to show that (\ref{al:3}) holds for the matrix $\widehat{\mathbf R}_n\left(l\right)$.

Let $\widetilde x_{uv}=\widehat x_{uv}-{\rm E}\left(\widehat x_{uv}\right)$, $\widetilde{\mathbf X}_n=\left(\widetilde x_{uv}\right)$, and
$$\widetilde {\mathbf R}_n\left(l\right)=n^{-l}{\widetilde{\mathbf X}_n\diamond\widetilde{\mathbf X}_n^*\diamond\cdots\diamond\widetilde{\mathbf X}_n}\diamond\widetilde{\mathbf X}_n^*.$$
Suppose (\ref{al:3}) is true for the matrix $\widetilde {\mathbf R}_n\left(l\right)$, then we assert that, for all $l\ge0$,
\begin{align}\label{al:4}
\left\|\widetilde {\mathbf R}_n\left(l\right)-\widehat {\mathbf R}_n\left(l\right)\right\|_2=0 \quad { a.s..}
\end{align}
In fact, $\widetilde {\mathbf R}_n\left(l\right)-\widehat {\mathbf R}_n\left(l\right)$ can be written as a sum of $\diamond$ products of matrices $\frac{1}{\sqrt n}\widetilde{\mathbf X}_n,{\rm E}\left(\frac{1}{\sqrt n}\widehat{\mathbf X}_n\right)$ or their complex conjugate transpose.
In each product, at least one of them is ${\rm E}\left(\frac{1}{\sqrt n}\widehat{\mathbf X}_n\right)$ or its complex conjugate transpose. Next, we estimate the bounds of $\left\|\frac{1}{\sqrt n}\widetilde{\mathbf X}_n\right\|_2$ and $\left\|{\rm E}\left(\frac{1}{\sqrt n}\widehat{\mathbf X}_n\right)\right\|_2$. If (\ref{al:3}) is true for the matrix $\widetilde {\mathbf R}_n\left(l\right)$, we have
\begin{align}\label{al:18}
&\limsup\left\|n^{-1/2}\widetilde{\mathbf X}_n\right\|_2^2\notag\\
=&\limsup\left\|\widetilde{\mathbf R}_n\left(1\right)+{\rm diag}\left(\frac{1}{n}\sum_{v}\left\|\widehat x_{uv}-{\rm E}\left(\widehat x_{uv}\right)\right\|^2I_2,u\le p\right)\right\|_2\notag\\
\le &6+\limsup\frac{1}{n}\max_{u\le p}\sum_{v=1}^n\left[\left\|\widehat x_{uv}\right\|^2+2\left\|\widehat x_{uv}\right\|\left\|{\rm E}\widehat x_{uv}\right\|+\left\|{\rm E}\widehat x_{uv}\right\|^2\right]
\end{align}
Denoting $h_v=\left\|\widehat x_{uv}\right\|^2-{\rm E}\left\|\widehat x_{uv}\right\|^2$ and using the fact that for all $k\geq 1$, $k!\geq\left(k/3\right)^k$, we have
\begin{align}
{\rm E}\left(\frac{1}{n}\sum\limits_{v=1}^n h_v\right)^{2m} &=\frac{1}{ n^{2m}}\sum\limits_{m_1+\ldots+m_n=2m}\frac{2m!}{{m_1}!\ldots{m_n}!}{\rm E}h_1^{m_1}\ldots {\rm E}h_n^{m_n} \notag\\
&\leq \frac{1}{ n^{2m}}\sum\limits_{k=1}^m \sum_{\substack{m_1+\ldots +m_k=2m \\ m_j\geq 2}}\frac{2m!}{k!m_1! \ldots m_k!}\prod_{j=1}^k \left(\sum_{v=1}^n {\rm E}h_v^{m_j}\right)\notag
 \\
&\leq  \sum_{k=1}^m  n^{-2m} k^{2m}\left(k!\right)^{-1}\left(2\delta_n^2 n\right)^{2m-2k}M^k n^k \notag\\
&\leq \sum_{k=1}^m\left(\frac{3Mk}{n}\right)^k\left({4\delta_n^4k^2}\right)^{m-k}\notag\\
&\leq \sum_{k=1}^m\left(\frac{3Mm}{n}\right)^k\left({4\delta_n^4k^2}\right)^{m-k}.\label{eqm1}
\end{align}
Select $m=\left[\log n\right]$ and let $f\left(k\right)=k\log\left(3Mm/n\right)+\left(m-k\right)\log\left(4\delta_n^4k^2\right)$, then $f'\left(k\right)$ (the derivative of $f\left(k\right)$) is
\begin{align*}
&\log\left(3Mm/n\right)-\log\left(4\delta_n^4k^2\right)+2\left(m-k\right)/k\\
=&\log\left(3Mm/\left(4n\delta_n^4\right)\right)-2\log k+2\left(m-k\right)/k\\
\le &-\frac{3}{4}\log n-2\log k+2\left(m-k\right)/k.
 \end{align*}
 We claim that the maximum term on the right hand side of (\ref{eqm1}) can only be $k=1$ or $2$. In fact, when $k>2$,
 \begin{align*}
 f'\left(k\right)\le&-\frac{3}{4}\log n-2\log k+2\left(m-k\right)/3\\
 \le&-\frac{1}{12}\log n-2\log k<0.
 \end{align*}
 Thus, we obtain for any fixed $t>0$
 \begin{align*}
 {\rm E}\left(\frac{1}{n}\sum\limits_{v=1}^n h_v\right)^{2m} &\le m\left(\frac{3M}{n}\left(4\delta_n^4\right)^{m-1}+\left(\frac{6M}{n}\right)^2\left(16\delta_n^4\right)^{m-2}\right)\\
 \le&\delta_n^{4m}=o\left(n^{-t}\right).
 \end{align*}
From the inequality above with $t >2$ and  Borel-Cantelli lemma, we have
\begin{align}\label{al:19}
\max_{u\le p}\left|\frac{1}{n}\sum_v^{}\left(\left\|\widehat x_{uv}\right\|^2-{\rm E}\left\|\widehat x_{uv}\right\|^2\right)\right|\to 0, \ { a.s..}
\end{align}
Thus, (\ref{al:18}) turns into
\begin{align*}
&\limsup\left\|n^{-1/2}\widetilde{\mathbf X}_n\right\|_2^2\notag\\
\le &6+\limsup\frac{1}{n}\max_{u\le p}\sum_{v=1}^n\left[{\rm E}\left\| x_{uv}\right\|^2+2\delta_n\sqrt n\left\|{\rm E}\widehat x_{uv}\right\|\right]\\
\le &6+\limsup\frac{1}{n}\max_{u\le p}\sum_{v=1}^n\left[{\rm E}\left\| x_{uv}\right\|^2+2{\rm E}\left\| x_{uv}\right\|^2\right]\\
\le&6+3=9.
\end{align*}
 And
\begin{align*}
\left\|{\rm E}\left(\frac{1}{\sqrt n}\widehat{\mathbf X}_n\right)\right\|_2\le\sqrt{\frac{1}{n}\max_{u\le p}\sum_{v}\left\|{\rm E}\widehat x_{uv}\right\|^2}\le\frac{M}{\delta_n\sqrt n}=o\left(1\right).
\end{align*}
Combining the above with Lemma \ref{le:5}, the proof of (\ref{al:4}) is complete. Therefore it suffices to show that (\ref{al:3}) for the matrix $\widetilde{\mathbf R}_n\left(l\right)$ is true.

For
brevity, we still use ${\mathbf R}_n\left(l\right)$ and $x_{uv}$ to denote the matrix and variables after
truncation and recentralization. We further assume that:
\begin{equation}\label{eq:1}
\begin{split}
&{\rm (i)}\quad{\rm E}\left(x_{uv}\right)=0,{\rm E}\left\|x_{uv}\right\|^2\le 1,\\
&{\rm (ii)}\quad{\rm E}\left\|x_{uv}\right\|^l\le \left(\delta_n\sqrt n\right)^{l-3}, \ {\rm for \ all \ }l\ge3.
\end{split}
\end{equation}

We will complete the proof under the additional conditions (\ref{eq:1}). Select a sequence of even integers $m$ with the properties $m/\log n\to\infty$ and $m\delta^{1/3}/\log n\to 0$. For any $\eta >\left(2l+1\right)\left(l +1\right)y^{\left(l-1\right)/2}$, we have
\begin{align}\label{al:8}
{\rm P}\left(\left\|\mathbf R_n\left(l\right)\right\|_2\ge\eta\right)\le \eta^{-2m}{\rm E}\left(\left\|\mathbf R_n\left(l\right)\right\|_2^{2m}\right)\le\eta^{-2m}{\rm E}{\rm tr}\mathbf R_n^{2m}\left(l\right).
\end{align}
We only need to estimate
\begin{align*}
{\rm E}{\rm tr}\mathbf R_n^{2m}\left(l\right)=n^{-2ml}\sum{\rm Etr}\left(x_{i_1j_1}
\overline x_{i_2j_1}x_{i_2j_2}\overline x_{i_3j_2}\cdots x_{i_{2ml}j_{2ml}}\overline x_{i_1j_{2ml}}\right)
\end{align*}
where the summation runs over all integers $i_1,\cdots,i_{2ml}$ from $\left\{1,2,\cdots,p\right\}$ and $j_1,\cdots,j_{2ml}$ from $\left\{1,2,\cdots,n\right\}$ subject to the conditions that, for any $\alpha=0,1,\cdots,2m-1$,
\begin{equation}\label{al:5}
\begin{split}
&i_{\alpha l+1}\neq i_{\alpha l+2},i_{\alpha l+2}\neq i_{\alpha l+3},\cdots,i_{\left(\alpha +1\right)l}\neq i_{\left(\alpha +1\right)l+1};\\
&j_{\alpha l+1}\neq j_{\alpha l+2},j_{\alpha l+2}\neq j_{\alpha l+3},\cdots,j_{\left(\alpha +1\right)l-1}\neq i_{\left(\alpha+1\right)l}.
\end{split}
\end{equation}
Defining graphs in accordance with the last section, the equality above can be rewritten as
\begin{small}
\begin{align}\label{al:6}
{\rm E}\left({\rm tr}\mathbf R_n^{2m}\left(l\right)\right)=n^{-2ml}\sum_{\mathbf{G}}\sum_{\mathbf{\mathcal{I},\mathcal{J}}}{\rm Etr}\left(x_{i_1j_1}
\overline x_{i_2j_1}x_{i_2j_2}\overline x_{i_3j_2}\cdots x_{i_{2ml}j_{2ml}}\overline x_{i_1j_{2ml}}\right)
\end{align}
\end{small}
\noindent
where $\mathbf{G}$ runs over all canonical graphs and $\mathbf G\left(\mathcal{I,J}\right)$ runs over the given isomorphic class. Obviously, if $\mathbf G$ has a single edge, the terms corresponding to this graph are zero. Thus, we need only to estimate the sum of all those terms whose $\mathbf G$ has no single edge.

Now, we begin to estimate the right-hand side of (\ref{al:6}).
Noticing that $x_{i_1j_1}
\overline x_{i_2j_1}x_{i_2j_2}\overline x_{i_3j_2}\cdots x_{i_{2ml}j_{2ml}}\overline x_{i_1j_{2ml}}$ can be written as $\left(\begin{array}{cc}
\alpha&\beta\\
-\bar\beta&\bar \alpha
\end{array}\right)$, and according to (\ref{al:11}), one has
\begin{align*}
&\left|{\rm tr}\left(x_{i_1j_1}
\overline x_{i_2j_1}x_{i_2j_2}\overline x_{i_3j_2}\cdots x_{i_{2ml}j_{2ml}}\overline x_{i_1j_{2ml}}\right)\right|\\
=&\left|\alpha+\bar\alpha\right|\le2\left(\left|\alpha\right|^2+\left|\beta\right|^2\right)^{1/2}\\
=&2\left\{\det{\left(\begin{array}{cc}
\alpha&\beta\\
-\bar\beta&\bar \alpha
\end{array}\right)}\right\}^{1/2}\\
\le&2\left\|{\rm det}\left(x_{i_1j_1}\right)
{\rm det}\left(x_{i_2j_1}\right)\cdots {\rm det}\left(x_{i_{2ml}j_{2ml}}\right){\rm det}\left(x_{i_1j_{2ml}}\right)\right\|\\
\le&2\left\|x_{i_1j_1}\right\|\left\|
 x_{i_2j_1}\right\|\cdots\left\| x_{i_{2ml}j_{2ml}}\right\|\left\|\overline x_{i_1j_{2ml}}\right\|
\end{align*}
which is similar to Lemma 3.6 in \cite{yin2013limit}.
Assume $r$ and $s$ be the number of up and down innovations, respectively.  Let $k=r+s$ denote the total number of innovations and $t$ denote the number of $\mathbf{T}_2$ edges.
Due to the inequality above, we get
\begin{align}\label{al:7}
&{\rm E}\left|{\rm tr}\left(x_{i_1j_1}
\overline x_{i_2j_1}x_{i_2j_2}\overline x_{i_3j_2}\cdots x_{i_{2ml}j_{2ml}}\overline x_{i_1j_{2ml}}\right)\right|\notag\\
\le&2{\rm E}\left\|x_{i_1j_1}\right\|\left\|
\overline x_{i_2j_1}\right\|\cdots\left\| x_{i_{2ml}j_{2ml}}\right\|\left\|\overline x_{i_1j_{2ml}}\right\|\notag\\
\le&2\left(\delta_n\sqrt n\right)^{4ml-2k-t}.
\end{align}

By Remark \ref{re:1}, we know that the number of graphs of each isomorphic class is less than $n^sp^{r+1}$. Thus, (\ref{al:6}) can be estimated by
 \begin{align}\label{al1}
{\rm E}\left({\rm tr}\mathbf R_n^{2m}\left(l\right)\right)\le 2n^{-2ml}\sum_{\mathbf{G}}n^sp^{r+1}\left(\delta_n\sqrt n\right)^{4ml-2k-t}.
 \end{align}
 In the following, we only consider the number of canonical graphs without single edges.
 Due to condition (\ref{al:5}), we split the graph $\mathbf G$ into $2m$ subgraphs $\mathbf G_1, \cdots,\mathbf G_{2m}$. Within each subgraph, except the first and the last edges, all edges do not coincide with their adjacent (prior to or behind) edges, that means, every ${\rm \mathbf{T}_1}$ edge must be followed by a
 ${\rm \mathbf{T}_4}$ or  ${\rm \mathbf{T}_1}$ edge, unless it is the last edge of $\mathbf G_j$. Let $a_j$ denote the number of pairs of consecutive edges $\left(t_1,t_4\right)$ in the subgraph $\mathbf G_j$ in which $t_1$ is a ${\rm \mathbf{T}_1}$ edge and $t_4$ is a ${\rm \mathbf{T}_4}$ edge. Then the number of consecutive innovations in $\mathbf G_j$ is not more than $a_j$ or $a_{j}+1$ (the latter
happens when the last edge of $\mathbf G_j$ is an innovation). Hence, the number of
ways to arrange the consecutive innovation sequences is not more than
\begin{align*}
\left(\begin{array}{cc}
2l\\
2a_j\end{array}\right)+
\left(\begin{array}{cc}
2l\\
2a_j+1\end{array}\right)=
\left(\begin{array}{cc}
2l+1\\
2a_j+1\end{array}\right).
\end{align*}
The number of the ways to select positions of edges (including $\mathbf{T}_1$, $\mathbf{T}_3$ and $\mathbf{T}_4$) is
\begin{align*}
\left(\prod_{j=1}^{2m}\left(\begin{array}{cc}
2l+1\\
2a_j+1\end{array}\right)\right)
\left(\begin{array}{cc}
4ml-k\\
k\end{array}\right).
\end{align*}

After fixing the positions of edges, we need to know the selections to plot an edge of the given type. For an innovation or an irregular $\mathbf{T}_3$ edge, there is only one way to plot once the subgraph prior to this edge is plotted. By Lemma \ref{le:1}, there are at most $t+1$ single innovations to be
matched by a regular $\mathbf{T}_3$ edge. By Lemma \ref{le:2}, there are at most $2t$ regular $\mathbf{T}_3$ edges.
Hence, there are at most $\left(t+1\right)^{2t}\le \left(t+1\right)^{2\left(4ml-2k\right)}$ ways to plot the regular $\mathbf{T}_3$-edges. For each $\mathbf{T}_4$ edge, there are at most $\left(k+1\right)^2$ ways to determine its two vertices. Therefore, there are at most \begin{small}$\left(\begin{array}{cc}\left(k+1\right)^2\\t\end{array}\right)$ \end{small} ways to plot the $t$ $\mathbf{T}_2$ edges. And, there are at most $t^{4ml-2k}<\left(t+1\right)^{4ml-2k}$ ways to distribute the $4ml-2k \  \mathbf{T}_4$ edges.

Together with the analysis above and (\ref{al:7}), (\ref{al1}) can be estimated by
\begin{align*}
{\rm E}\left({\rm tr}\mathbf R_n^{2m}\left(l\right)\right)\le &2n^{-2ml}\sum\left(\prod_{j=1}^{2m}\left(\begin{array}{cc}
2l+1\\
2a_j+1\end{array}\right)\right)
\left(\begin{array}{cc}
4ml-k\\
k\end{array}\right)
\left(\begin{array}{cc}
\left(k+1\right)^2\\
t\end{array}\right)\\
&¡¤\left(t+1\right)^{3\left(4ml-2k\right)}
n^sp^{r+1}\left(\delta_n\sqrt n\right)^{4ml-2k-t}\\
\le2 &n\sum\left(\prod_{j=1}^{2m}\left(\begin{array}{cc}
2l+1\\
2a_j+1\end{array}\right)\right)
\left(\begin{array}{cc}
4ml-k\\
k\end{array}\right)
\left(\begin{array}{cc}
\left(k+1\right)^2\\
t\end{array}\right)\\
&¡¤\left(t+1\right)^{3\left(4ml-2k\right)}
y_n^{r+1}\delta_n^{4ml-2k}\left(\delta_n\sqrt n\right)^{-t}
\end{align*}
where the summation is taken subject to restrictions $1 \le k \le 2ml, 0 \le t \le 2ml, \ {\rm and} \  0\le a_j\le l$. Applying (5.2.16) and (5.2.17) in \cite{bai2010spectral}, i.e.
\begin{align*}
y_n^{r+1}\le y_n^{\left(k-t-2m\right)/2}
\end{align*}
and
\begin{align*}
\left(\begin{array}{cc}
2l+1\\
2a_j+1\end{array}\right)
\le\left(2l+1\right)^{2\sum a_j+2m}\le\left(2l+1\right)^{2t+2m},
\end{align*}
we have
\begin{align*}
{\rm E}\left({\rm tr}\mathbf R_n^{2m}\left(l\right)\right)
\le 2&n\sum_{k=1}^{2ml}\sum_{t=0}^{4ml-2k}\left(2l+1\right)^{2m}\left(l+1\right)^{2m}
\left(\begin{array}{cc}
4ml-k\\
k\end{array}\right)\\
&\left(\frac{\left(k+1\right)^2}{\delta_n\sqrt n}\right)^{-t}
\left(t+1\right)^{3\left(4ml-2k\right)}y_n^{\left(k-t-2m\right)/2}\delta_n^{4ml-2k}\\
\le 2&n^2\left(2l+1\right)^{2m}\left(l+1\right)^{2m}y_n^{-m}\sum_{k=1}^{2ml}
\left(\begin{array}{cc}
4ml-k\\
k\end{array}\right)\\
&\left(\frac{3\left(4ml-2k\right)\delta_n^{1/3}}{\log{\delta_n\sqrt {y_nn}/\left(k+1\right)^2}}\right)^{3\left(4ml-2k\right)}
y_n^{k/2}\\
\le 2&n^2\left(2l+1\right)^{2m}\left(l+1\right)^{2m}y_n^{-m}
\left[y_n^{1/4}+\frac{24ml\delta_n^{1/3}}{\frac{1}{3}\log n}\right]^{4ml}\\
\le 2&n^2\left(2l+1\right)^{2m}\left(l+1\right)^{2m}y_n^{m\left(l-1\right)}
\left[1+o\left(1\right)\right]^{4ml}
\end{align*}
where the second inequality follows from the elementary inequality
\begin{align*}
\alpha^{-(t+1)}\left(t+1\right)^{\beta}\le\left(\frac{\beta}{\log (\alpha)}\right)^{\beta},\beta>0,\alpha>1.
\end{align*}
Thus, combining the inequalities above, $m/\log n\to \infty$ with (\ref{al:8}), we have
\begin{align*}
{\rm P}\left(\left\|\mathbf R_n^{2m}\left(l\right)\right\|_2\ge\eta\right)\le2n^2\left(\frac{\left(2l+1\right)\left(l+1\right)y_n^{\left(l-1\right)}}{\eta}\right)^{2m}
\left[1+o\left(1\right)\right]^{4ml}
\end{align*}
which is summable. Therefore, by Borel-Cantelli lemma, one has
\begin{align*}
\limsup_{n\to\infty}\left\|\mathbf R_n\left(l\right)\right\|_2\le\left(2l+1\right)\left(l+1\right)y^{\left(l-1\right)/2}
 \ a.s..
\end{align*}
\end{proof}
In the following, we say that a matrix is $o\left(1\right)$ if its $2$-norm tends to $0$.
\begin{lemma}\label{le:6}
Under the conditions of Theorem \ref{th:1}, we have
\begin{align}\label{al:12}
\mathbf R_n\mathbf R_n\left(k\right)=\mathbf R_n\left(k+1\right)+y\sigma^2\mathbf R_n\left(k\right)+y\sigma^4\mathbf R_n\left(k-1\right)+o\left(1\right) \ { a.s..}
\end{align}
\end{lemma}
\begin{proof}
We only need to show that (\ref{al:12}) is true for $\sigma=1$. Define $\mathbf X_n^{\left(3\right)}=n^{-\frac{3}{2}}\left(\left\|x_{uv}\right\|^2x_{uv}\right)$, then, by (\ref{al:19}), we get
\begin{align}\label{al:13}
\left\|\mathbf X_n^{\left(3\right)}\right\|_2^2\le& n^{-3}\max_{u\le p}\sum_v\left\|x_{uv}\right\|^6
\le\delta_n^4\frac{1}{n}\max_{u\le p}\sum_v\left\|x_{uv}\right\|^2\\
\le&\delta_n^4\frac{1}{n}\max_{u\le p}\sum_v{\rm E}\left\|x_{uv}\right\|^2
\le\delta_n^4
\to0 \ {\rm a.s..}
\end{align}
According to Definition \ref{de:1} and (\ref{al:13}),
\begin{align}
\mathbf R_n\left(k\right)=&n^{-k}\overbrace{\mathbf X_n\diamond\mathbf X_n^*\diamond\cdots\diamond\mathbf X_n}^{k \ \mathbf X_n}\diamond\mathbf X_n^*\notag\\
=&n^{-k}\mathbf X_n\left(\overbrace{\mathbf X_n^*\diamond\cdots\diamond\mathbf X_n\diamond\mathbf X_n^*}^{k \ \mathbf X_n^*}\right)
-n^{-1}\left[{\rm diag}\left(\mathbf X_n\mathbf X_n^*\right)\right]\mathbf R_n\left(k-1\right)\notag\\
+&\mathbf X_n^{\left(3\right)}\diamond\left(n^{-\frac{2k-3}{2}}\overbrace{\mathbf X_n^*\diamond\cdots\diamond\mathbf X_n\diamond\mathbf X_n^*}^{k-1 \ \mathbf X_n^*}\right)\notag\\
=&n^{-k}\mathbf X_n\left(\overbrace{\mathbf X_n^*\diamond\cdots\diamond\mathbf X_n\diamond\mathbf X_n^*}^{k \ \mathbf X_n^*}\right)
-\mathbf R_n\left(k-1\right)+o\left(1\right) \ a.s.
\end{align}
where the last equality follows from Lemma \ref{le:5}. Similarly
\begin{align}
\mathbf R_n\left(k+1\right)=&n^{-k-1}\mathbf X_n\left(\overbrace{\mathbf X_n^*\diamond\cdots\diamond\mathbf X_n\diamond\mathbf X_n^*}^{k+1 \ \mathbf X_n^*}\right)
-n^{-1}\left[{\rm diag}\left(\mathbf X_n\mathbf X_n^*\right)\right]\mathbf R_n\left(k\right)\notag\\
+&o\left(1\right) \ a.s.\notag\\
=&n^{-1}\mathbf X_n\mathbf X_n^*\mathbf R_n\left(k\right)-n^{-k-1}\mathbf X_n{\rm diag}\left(\mathbf X_n^*\mathbf X_n\right)\left(\overbrace{\mathbf X_n^*\diamond\cdots\diamond\mathbf X_n\diamond\mathbf X_n^*}^{k \ \mathbf X_n^*}\right)\notag\\
-&n^{-1}\left[{\rm diag}\left(\mathbf X_n\mathbf X_n^*\right)\right]\mathbf R_n\left(k\right)
+o\left(1\right) \ a.s.\notag\\
=&\mathbf R_n\mathbf R_n\left(k\right)-yn^{-k}\mathbf X_n\left(\overbrace{\mathbf X_n^*\diamond\cdots\diamond\mathbf X_n\diamond\mathbf X_n^*}^{k \ \mathbf X_n^*}\right)
+o\left(1\right) \ a.s.\notag\\
=&\mathbf R_n\mathbf R_n\left(k\right)-y\left(\mathbf R_n\left(k\right)+\mathbf R_n\left(k-1\right)\right)
+o\left(1\right) \ a.s..
\end{align}
The proof is complete.
\end{proof}
\begin{lemma}\label{le:7}
Under the conditions of Theorem \ref{th:1}, we have
\begin{small}
\begin{align}\label{al:14}
\left(\mathbf R_n-y\sigma^2\mathbf I_{p}^Q\right)^k=\sum_{r=0}^k\left(-1\right)^{r+1}\sigma^{2\left(k-r\right)}\mathbf R_n\left(r\right)\sum_{j=0}^{\left[\left(k-r\right)/2\right]}C_j\left(k,r\right)y^{k-r-j}+o\left(1\right)
\end{align}
\end{small}
where the constants $\left|C_j\left(k,r\right)\right|\le2^k$.
\end{lemma}
\begin{proof}
We shall prove this lemma by induction on $k$.
\begin{itemize}
\item{1.} When $k=1$,
 $$\mathbf R_n-y\sigma^2\mathbf I_{p}^Q=\mathbf R_n\left(1\right)C_0\left(1,1\right)-y\sigma^2\mathbf R_n\left(0\right)C_0\left(1,0\right)$$
 where $C_0\left(1,1\right)=1$  and $C_0\left(1,0\right)=1$.
\item{2.} Suppose the lemma is true for $k$. By Lemma \ref{le:6}, we have
\begin{align*}
&\left(\mathbf R_n-y\sigma^2\mathbf I_{p}^Q\right)^{k+1}\\
=&\left(\mathbf R_n-y\sigma^2\mathbf I_{p}^Q\right)\bigg(\sum_{r=0}^k\left(-1\right)^{r+1}\sigma^{2\left(k-r\right)}\mathbf R_n\left(r\right)\\
\times&\sum_{j=0}^{\left[\left(k-r\right)/2\right]}C_j\left(k,r\right)y^{k-r-j}+o\left(1\right)\bigg)\\
=&\mathbf R_n\bigg(\sum_{r=1}^k\left(-1\right)^{r+1}\sigma^{2\left(k-r\right)}\mathbf R_n\left(r\right)\sum_{j=0}^{\left[\left(k-r\right)/2\right]}C_j\left(k,r\right)y^{k-r-j}\\
-&\sigma^{2k}\sum_{j=0}^{\left[k/2\right]}C_j\left(k,0\right)y^{k-j}\mathbf I_{p}^Q\bigg)\\
-&\sum_{r=0}^k\left(-1\right)^{r+1}\sigma^{2\left(k-r+1\right)}\mathbf R_n\left(r\right)\sum_{j=0}^{\left[\left(k-r\right)/2\right]}C_j\left(k,r\right)y^{k-r-j+1}+o\left(1\right)\\
=&\sum_{r=1}^k\left(-1\right)^{r+1}\sigma^{2\left(k-r\right)}\left(\mathbf R_n\left(r+1\right)+y\sigma^2\mathbf R_n\left(r\right)+y\sigma^4\mathbf R_n\left(r-1\right)\right)\\
\times&\sum_{j=0}^{\left[\left(k-r\right)/2\right]}C_j\left(k,r\right)y^{k-r-j}
-\sigma^{2k}\sum_{j=0}^{\left[k/2\right]}C_j\left(k,0\right)y^{k-j}\mathbf R_n\\
-&\sum_{r=0}^k\left(-1\right)^{r+1}\sigma^{2\left(k-r+1\right)}\mathbf R_n\left(r\right)\sum_{j=0}^{\left[\left(k-r\right)/2\right]}C_j\left(k,r\right)y^{k-r-j+1}+o\left(1\right)\\
=&\sum_{r=1}^{k+1}\left(-1\right)^{r+1}\sigma^{2\left(k+1-r\right)}\mathbf R_n\left(r\right)\sum_{j=0}^{\left[\left(k+1-r\right)/2\right]}\left[-C_j\left(k,r-1\right)\right]y^{k+1-r-j}\\
+&\sum_{r=0}^{k-1}\left(-1\right)^{r+1}\sigma^{2\left(k+1-r\right)}\mathbf R_n\left(r\right)\sum_{j=1}^{\left[\left(k+1-r\right)/2\right]}\left[-C_j\left(k,r+1\right)\right]y^{k+1-r-j}\\
-&\sigma^{2\left(k+1\right)}\sum_{j=0}^{\left[k/2\right]}C_{j-1}\left(k,0\right)y^{k-j}\mathbf I_{p}^Q+o\left(1\right)\\
=&\sum_{r=0}^{k+1}\left(-1\right)^{r+1}\sigma^{2\left(k-r\right)}\mathbf R_n\left(r\right)\sum_{j=0}^{\left[\left(k+1-r\right)/2\right]}C_j\left(k+1,r\right)y^{k+1-r-j}+o\left(1\right) \ a.s.
\end{align*}
where $C_j\left(k+1,r\right)$ is a sum of one or two terms of the form $-C_j\left(k,r-1\right)$ and $-C_j\left(k,r+1\right)$.
\item{3.} By induction,we conclude that (\ref{al:14}) is true for all fixed $k$.
\end{itemize}
Thus, the proof of this lemma is complete.
\end{proof}

\subsection{Proof of Theorem \ref{th:1}}
By Lemma \ref{le:3} and Lemma \ref{le:7}, for any fixed $k$, we have
\begin{align*}
\left\|\mathbf R_n-y\sigma^2\mathbf I_{p}^Q\right\|_2^k\le&\sum_{r=0}^k\sigma^{2\left(k-r\right)}\left\|\mathbf R_n\left(r\right)\right\|_2\sum_{j=0}^{\left[\left(k-r\right)/2\right]}C_j\left(k,r\right)y^{k-r-j}\\
\le&\sum_{r=0}^k\sigma^{2k}\left(2r+1\right)\left(r+1\right)y^{\frac{r-1}{2}}\left[\left(k-r\right)/2\right]2^ky^{\frac{k-r}{2}}\\
\le&C\sigma^{2k}k^42^ky^{\frac{k-1}{2}}.
\end{align*}
Therefore,
\begin{align*}
\left\|\mathbf R_n-y\sigma^2\mathbf I_{p}^Q\right\|_2\le C^{1/k}\sigma^{2}k^{4/k}2y^{\frac{k-1}{2k}}.
\end{align*}
Letting $k\to\infty$, we obtain
\begin{align}\label{al:15}
\limsup\left\|\mathbf R_n-y\sigma^2\mathbf I_{p}^Q\right\|_2\le2\sigma^2\sqrt y \ a.s..
\end{align}
By (\ref{al:19}), we have
\begin{align}\label{al:16}
\left\|\mathbf S_n-\sigma^2\mathbf I_{p}^Q-\mathbf R_n\right\|_2=&\left\|{\rm diag}\left(\mathbf S_n\right)-\sigma^2\mathbf I_{p}^Q\right\|_2\notag\\
\le& \max_{u\le p}\left|\frac{1}{n}\sum_{v=1}^n\left(\left\|x_{uv}\right\|^2-\sigma^2\right)\right|\to 0 \ a.s..
\end{align}
Together with (\ref{al:15}) and (\ref{al:16}), one has
\begin{align*}
\left\|\mathbf S_n-\sigma^2\left(1+y\right)\mathbf I_{p}^Q\right\|_2\le\left\|\mathbf S_n-\sigma^2\mathbf I_{p}^Q-\mathbf R_n\right\|_2+\left\|\mathbf R_n-y\sigma^2\mathbf I_{p}^Q\right\|_2
\le2\sigma^2\sqrt y \ a.s..
\end{align*}
The proof of Theorem \ref{th:1} is complete.

\section{Proof of Theorem \ref{th:2}}
Due to Theorem 1.1 in Li, Bai and Hu \cite{li2013convergence}, with probability 1, we have
$$\limsup s_{\min}\left(\mathbf S_n\right)\le\sigma^2\left(1-\sqrt y\right)^2$$
and
$$\liminf s_p\left(\mathbf S_n\right)\ge\sigma^2\left(1+\sqrt y\right)^2.$$
Then, by Theorem \ref{th:1},
\begin{align*}
\limsup s_p\left(\mathbf S_n\right)=&\sigma^2\left(1+y\right)+\limsup s_p\left(\mathbf S_n-\sigma^2\left(1+y\right)\mathbf I_{p}^Q\right)\\
\le&\sigma^2\left(1+y\right)+2\sigma^2\sqrt y=\sigma^2\left(1+\sqrt y\right)^2
\end{align*}
and
\begin{align*}
\liminf s_{\min}\left(\mathbf S_n\right)\ge&\sigma^2\left(1+ y\right)+\liminf s_{\min}\left(\mathbf S_n-\sigma^2\left(1+y\right)\mathbf I_{p}^Q\right)\\
\ge&\sigma^2\left(1+y\right)-2\sigma^2\sqrt y=\sigma^2\left(1-\sqrt y\right)^2.
\end{align*}
Combining the above inequalities, we get $a.s.$
$$\lim s_{\min}\left(\mathbf S_n\right)=\left(1-\sqrt y\right)^2\sigma^2$$
and
$$\lim s_p\left(\mathbf S_n\right)=\left(1+\sqrt y\right)^2\sigma^2.$$
Therefore, we conclude the proof of Theorem \ref{th:2}.

\section{Proof of Theorem \ref{th:3}}
By Remark \ref{re:2}, we only need to prove the necessity of the conditions.
\subsection{Condition $\left({\rm i}\right)$}
Define a unit vector $\mathbf e_{2u}=\left(0,\cdots,0,1,0,\cdots,0\right)^{\prime},u=1,\cdots,p$, then,
\begin{align*}
s_p\left(\mathbf Z_n\right)=s_{2p}\left(\psi\left(\mathbf Z_n\right)\right)\ge\mathbf e_{2u}^{\prime}\psi\left(\mathbf Z_n\right)\mathbf e_{2u}=\frac{1}{n}\sum_{v=1}^n\left\|z_{uv}\right\|^2
\end{align*}
which implies that
\begin{align*}
s_p\left(\mathbf Z_n\right)\ge\max_{u\le p}\frac{1}{n}\sum_{v=1}^n\left\|z_{uv}\right\|^2.
\end{align*}
By Lemma \ref{le:4}, if ${\rm E}\left\|z_{11}\right\|^4=\infty$, we obtain
\begin{align*}
\limsup_{n\to\infty}\max_{u\le p}\frac{1}{n}\sum_{v=1}^n\left\|z_{uv}\right\|^2\to\infty \ { a.s..}
\end{align*}
This contradicts the assumptions. The condition $\left({\rm i}\right)$ is proved.

\subsection{Condition $\left({\rm iii}\right)$}
Suppose ${\rm E}\left\|z_{11}\right\|^4<\infty$ but ${\rm E}\left(z_{11}\right)=\hbar\neq 0$ ($\hbar$ is a quaternion). Then
\begin{align*}
\left\|\frac{1}{\sqrt n}\mathbf Z_n\right\|_2\ge&\left\|\frac{1}{\sqrt n}\left({\rm E}\mathbf Z_n\right)\right\|_2-\left\|\frac{1}{\sqrt n}\left(\mathbf Z_n-{\rm E}\mathbf Z_n\right)\right\|_2\\
\ge&\frac{\left\|\hbar\right\|\sqrt{pn}}{\sqrt n}-\left\|\frac{1}{\sqrt n}\left(\mathbf Z_n-{\rm E}\mathbf Z_n\right)\right\|_2\to\infty, \ { a.s..}
\end{align*}
This contradicts the assumptions. The condition $\left({\rm iii}\right)$ is proved.
\subsection{The completion of Theorem \ref{th:3}}
Conditions $\left({\rm ii}\right)$ and $\left({\rm iv}\right)$ follow from Theorem \ref{th:1}.  Thus, the proof of the theorem is complete.

\section{Appendix}
In this section, we list some  lemmas for readers convenience.

\begin{lemma}[Lemma B.25 in \cite{bai2010spectral}]\label{le:4}
Let $\left\{z_{jk},j,k=1,2,\cdots\right\}$ be a double array of i.i.d. complex random variables and let $\alpha>\frac{1}{2},\beta\ge0$, and $M>0$ be constants. Then, as $n\to\infty$,
\begin{align*}
\max_{j\le Mn^{\beta}}\left|n^{-\alpha}\sum_{k=1}^{n}\left(z_{jk}-c\right)\right|\to 0,a.s.
\end{align*}
if and only if the following hold:

(i) ${\rm E}\left|z_{11}\right|^{\left(1+\beta\right)/\alpha}<\infty$;

(ii) $c=\left\{\begin{array}{cc}{\rm E}\left(z_{11}\right),&{\rm if} \ \alpha\le1,\\{\rm any \  number},&{\rm if} \ \alpha>1.\end{array}\right.$
\end{lemma}

\begin{lemma}[Lemma A.11 in \cite{bai2010spectral}]\label{le:8}
Let $\mathbf A$ be an $m\times n$ matrix with singular values $s_j\left(\mathbf A\right)$,$j=1,2,\cdots,q=\min\left\{m,n\right\}$,arranged in decreasing order. Then, for any integer $k\left(1\le k\le q\right)$,
\begin{align*}
\sum_{j=1}^ks_j\left(\mathbf A\right)=\sup_{\mathbf{E^*E=F^*F=I_k}}\left|{\rm tr}\left(\mathbf{E^*AF}\right)\right|,
\end{align*}
where the orders of $\mathbf E$ are $m\times k$ and those of $\mathbf F$ are $n\times k$.
\end{lemma}

\begin{thebibliography}{30}

\bibitem{adler1995quaternionic}
S.~L. Adler.
\newblock {\em Quaternionic quantum mechanics and quantum fields}, volume~1.
\newblock Oxford University Press Oxford, 1995.

\bibitem{bai1999methodologies}
Z.~D. Bai.
\newblock Methodologies in spectral analysis of large-dimensional random
  matrices, a review.
\newblock {\em Statist. Sinica}, 9(3):611--677, 1999.

\bibitem{bai1998no}
Z.~D.~Bai and J.~W. Silverstein.
\newblock No eigenvalues outside the support of the limiting spectral
  distribution of large-dimensional sample covariance matrices.
\newblock {\em The Annals of Probability}, 26(1):316--345, 1998.

\bibitem{bai2010spectral}
Z.~D.~Bai and J.~W. Silverstein.
\newblock {\em Spectral analysis of large dimensional random matrices}.
\newblock Springer, 2010.

\bibitem{Bai1988166}
Z.~D. Bai, J.~W. Silverstein, and Y.~Yin.
\newblock A note on the largest eigenvalue of a large dimensional sample
  covariance matrix.
\newblock {\em Journal of Multivariate Analysis}, 26(2):166 -- 168, 1988.

\bibitem{BaiYin1993}
Z.~D. Bai and Y.~Q. Yin.
\newblock Limit of the smallest eigenvalue of a large dimensional sample
  covariance matrix.
\newblock {\em The Annals of Probability}, 21(3):pp. 1275--1294, 1993.

\bibitem{bai1987limiting}
Z.~D. Bai, Y.~Q. Yin, and P.~R. Krishnaiah.
\newblock On the limiting empirical distribution function of the eigenvalues of
  a multivariate f matrix.
\newblock {\em Theory of Probability \& Its Applications}, 32(3):490--500,
  1987.

\bibitem{burkholder1973distribution}
D.~L. Burkholder.
\newblock Distribution function inequalities for martingales.
\newblock {\em the Annals of Probability}, 1(1):19--42, 1973.

\bibitem{finkelstein1962foundations}
D.~Finkelstein, J.~M. Jauch, S.~Schiminovich, and D.~Speiser.
\newblock Foundations of quaternion quantum mechanics.
\newblock {\em Journal of mathematical physics}, 3(2):207, 1962.

\bibitem{geman1980limit}
S.~Geman.
\newblock A limit theorem for the norm of random matrices.
\newblock {\em The Annals of Probability}, 8(2):252--261, 1980.

\bibitem{kuipers1999quaternions}
J.~B. Kuipers.
\newblock {\em Quaternions and rotation sequences}.
\newblock Princeton university press Princeton, 1999.

\bibitem{li2013convergence}
H. Q.~Li, Z.~D.~Bai, and J.~Hu.
\newblock Convergence of empirical spectral distributions of large dimensional
  quaternion sample covariance matrices.
\newblock {\em arXiv preprint arXiv:1310.5428}, 2013.

\bibitem{mehta2004random}
M.~L. Mehta.
\newblock {\em Random matrices}, volume 142.
\newblock Access Online via Elsevier, 2004.

\bibitem{silverstein1985smallest}
J.~W. Silverstein.
\newblock The smallest eigenvalue of a large dimensional wishart matrix.
\newblock {\em The Annals of Probability}, 13(4):1364--1368, 1985.

\bibitem{zhang1994numerical}
W.~So, R.~C. Thompson, and F.~Zhang.
\newblock The numerical range of normal matrices with quaternion entries.
\newblock {\em Linear and Multilinear Algebra}, 37(1-3):175--195, 1994.

\bibitem{yin2013limit}
Y.~Yin, Z. D.~Bai, and J.~Hu.
\newblock On the limit of extreme eigenvalues of large dimensional random
  quaternion matrices.
\newblock {\em arXiv preprint arXiv:1312.1433}, 2013.

\bibitem{Yin1988}
Y. Q.~Yin, Z. D.~Bai, and P.~Krishnaiah.
\newblock On the limit of the largest eigenvalue of the large dimensional
  sample covariance matrix.
\newblock {\em Probability Theory and Related Fields}, 78(4):pp. 509--521,
  1988.

\bibitem{zhang1995numerical}
F.~Zhang.
\newblock On numerical range of normal matrices of quaternions.
\newblock {\em J. Math. Physical Sci}, 29(6):235--251, 1995.

\bibitem{zhang1997quaternions}
F.~Zhang.
\newblock Quaternions and matrices of quaternions.
\newblock {\em Linear algebra and its applications}, 251:21--57, 1997.

\end{thebibliography}

\end{document}